\documentclass[11pt]{extarticle}

\usepackage[left=1in,right=1in,top=1in,bottom=1.3in]{geometry}
\usepackage{amsmath, amssymb, amsthm, latexsym, amsfonts, indentfirst, xcolor, mathtools, manfnt, microtype}
\usepackage[utf8]{inputenc}
\usepackage[english]{babel}

\usepackage[breaklinks]{hyperref}
\usepackage[nameinlink]{cleveref}
\hypersetup{
	colorlinks = true, % Colours links instead of ugly boxes
	urlcolor = cyan, % Colour for external hyperlinks
	linkcolor = teal, % Colour of internal links
	citecolor = cyan % Colour of citations
}

\usepackage{tikz} % for figures
\let\OLDthebibliography\thebibliography
\renewcommand\thebibliography[1]{
	\OLDthebibliography{#1}
	\setlength{\parskip}{0pt}
	\setlength{\itemsep}{0pt plus 0.3ex}
}

\newtheorem{Theorem}{Theorem}
\Crefname{Theorem}{Theorem}{Theorems}

\newtheorem{Lemma}{Lemma}

\newtheorem{Proposition}{Proposition}

\newcommand{\R}{\mathbb R}
\newcommand{\Z}{\mathbb Z}
\newcommand{\N}{\mathbb N}

\newcommand{\B}{\mathcal B}
\newcommand{\F}{\mathcal F}
\newcommand{\M}{\mathcal M}

\newcommand{\e}{\mathbf e}
\newcommand{\SC}{\mathrm{SC}}
\newcommand{\U}{\mathbf{u}}
\newcommand{\VV}{V_1\times\dots\times V_k}
\newcommand{\W}{\mathbf{w}}
\newcommand{\x}{\mathbf{x}}
\newcommand{\y}{\mathbf{y}}

\newcommand{\cube}{\mbox{\mancube}}
\newcommand{\tetrahedron}{%
	\tikz[baseline]{
		\draw (0,0) -- (-0.05,0.15);
		\draw (-0.05,0.15) -- (-0.2,0.2);
		\draw (-0.05,0.15) -- (0.09,0.26);
		\draw (0,0) -- (0.09,0.26);
		\draw (-0.2,0.2) -- (0.09,0.26);
		\draw (-0.2,0.2) -- (0,0)
	}
}

\begin{document}

\date{}
\title{Cutting corners}
\author{
	Andrey Kupavskii\thanks{Moscow Institute of Physics and Technology,  St. Petersburg State University, Innopolis University, Russia. Email:~\href{mailto:kupavskii@ya.ru}{\tt kupavskii@ya.ru}.}
	\and
	Arsenii Sagdeev\thanks{Alfréd Rényi Institute of Mathematics, Budapest, Hungary and MIPT, Moscow, Russia. Email:~\href{mailto:sagdeevarsenii@gmail.com}{\tt sagdeevarsenii@gmail.com}.}
	\and
	Dmitrii Zakharov\thanks{ Department of Mathematics, Massachusetts Institute of Technology, Cambridge, MA 02139, USA.     
    Email:~\href{mailto:zakhdm@mit.edu}{\tt zakhdm@mit.edu}.} }

\maketitle

\begin{abstract}
	We say that a subset $\mathcal M$ of $\mathbb R^n$ is \textit{exponentially Ramsey} if there exists $\varepsilon>0$ and $n_0$ such that $\chi(\mathbb R^n, \mathcal M) \ge (1+\varepsilon)^n$ for any $n>n_0$, where $\chi(\mathbb R^n, \mathcal M)$ stands for the minimum number of colors in a coloring of $\R^n$ such that no copy of $\mathcal M$ is monochromatic. One important result in Euclidean Ramsey theory is due to Frankl and R\"odl, and states the following (under some mild extra conditions): if both $\mathcal{N}_1$ and $\mathcal{N}_2$ are exponentially Ramsey then so is their Cartesian product. Applied several times to simple two-point sets $\mathcal{N}_i$, this result implies that any subset $\mathcal M$ of a `hyperrectangle'  $\mathcal{N}_1\times\cdots\times \mathcal{N}_k$ is exponentially Ramsey.
	
	However, generally, such `embeddings' of $\mathcal M$ result in very inefficient bounds on the aforementioned $\varepsilon$. In this paper, we present another way of combining exponentially Ramsey sets, which gives much better estimates in some important cases. In particular, we show that the chromatic number of $\mathbb R^n$ with a forbidden equilateral triangle satisfies $\chi(\mathbb R^n, \triangle) \ge \big(1.0742...+o(1)\big)^n$, greatly improving upon the previous constant $1.0144$. We also obtain similar strong results  for regular simplices of larger dimensions, as well as for related geometric Ramsey-type questions in Manhattan norm.
	
	We then show that the same technique implies several interesting corollaries in other combinatorial problems. In particular, we give an explicit upper bound on the size of a family $\mathcal F \subset 2^{[n]}$ that contains no \textit{weak $k$-sunflowers}, i.e. no collection of $k$ sets with pairwise intersections of the same size. This bound improves upon previously known results for all $k\ge 4$. %Second, we prove that there are three pairwise orthogonal vectors \textcolor{orange}{TBA}.
	Finally, we also present a simple deduction of the (other) celebrated Frankl--R\"odl theorem from an earlier result of Frankl and Wilson. It gives probably the shortest known proof of Frankl and R\"odl result with the most efficient bounds.
	
	\vspace{2mm}
	%MSC-class:	05D10, 05D05, 05C15
	\end{abstract}

\section{Introduction} \label{sec1}

\subsection{Euclidean Ramsey Theory} \label{sec1:ERT}

Probably the most famous geometric coloring problem is to determine the chromatic number of the plane. For an $n$-dimensional Euclidean space $\R^n$, its {\it chromatic number} $\chi(\R^n)$ is the smallest $r$ such that there exists a coloring of the points of $\R^n$ with $r$ colors, i.e. an {\it $r$-coloring}, and with no two points of the same color at unit distance apart. Currently, the  best bounds on the plane are $5 \le \chi(\R^2) \le 7$, where the lower bound was obtained relatively recently, see~\cite{degrey,Exoo,Soifer}.

In an influential paper~\cite{EGMRSS1}, Erd\H os, Graham, Montgomery, Rothschild, Spencer and Straus laid the foundation of {\it Euclidean Ramsey theory}, a domain that studies different questions about monochromatic patterns in colorings of Euclidean spaces.  See a recent survey by Graham~\cite{Graham2017} on this topic. One of the central notions in this field is as follows: for a subset $\M \subset \R^d$, the {\it chromatic number $\chi(\R^n,\M)$} is defined as the smallest $r$ such that there exists an $r$-coloring of $\R^n$ with no monochromatic isometric copy of $\M$. We say that a set $\M$ is {\it Ramsey} if $\chi(\R^n, \M)\to \infty$ as $n\to \infty,$ and {\it exponentially Ramsey} if there exists some $\varepsilon>0$ such that\footnote{Throughout this paper, all the $o$-notation is with respect to $n\to \infty$, whenever the opposite is not stated explicitly.} $\chi(\R^n, \M)\ge \big(1+\varepsilon+o(1)\big)^n$.

Frankl and Wilson~\cite{FranklWil1981} showed that the simplest configuration $\M$, consisting of $2$ points at unit distance apart, is exponentially Ramsey using polynomial method. Actually, they showed a bit more, namely, that this configuration is \textit{super-Ramsey}. We call $\M\subset \R^d$ {\it super-Ramsey}, if for a sufficiently small $\varepsilon>0$ and for any dimension $n$, there is a set $V\subset \R^n$ of size that is at most exponential  in $n$ such that $|V|/\alpha(V, \M) \ge \big(1+\varepsilon+o(1)\big)^n$, where $\alpha(V,\M)$ stands for the maximum cardinality $|W|$ of a subset $W\!\subset\! V$ with no isometric copy of $\M$. This property is stronger than exponentially Ramsey. Just note that
\begin{equation} \label{eq:chi_alpha2}
	\chi(\R^n, \M) \ge \frac{|V|}{\alpha(V, \M)}
\end{equation}
by the pigeonhole principle. If $\M$  consists of $2$ points at unit distance apart, then we denote $\alpha(V,\M)$ simply by $\alpha(V)$ for a shorthand. The current record on the aforementioned value of $\varepsilon$ in this case belongs to Raigorodskii~\cite{Rai3} who managed to prove the following.

\begin{Theorem} \label{th:rai}
	There is a sequence of subsets $V(n) \subset \R^n$, $n \in \N$, such that
	$ \frac{|V(n)|}{\alpha(V(n))} \ge \big(\psi_2+o(1)\big)^n,$
	where $\psi_2 = \sup\limits_{\scriptscriptstyle 0\le x \le 1} \frac{1+x+x^3}{1+x^2+x^4} = 1.239...$
\end{Theorem}

More super-Ramsey configurations were found in~\cite{FranklRodl1987} by Frankl and R\"odl, who showed that the set of all $k+1$ vertices of the regular unit $k$-dimensional simplex $\triangle^k$ have this property for all $k \in \N$. In cases $k=2$ and $3$, we denote these sets by $\triangle$ and $\tetrahedron$ for a shorthand. Though Frankl and R\"odl did not provide explicit lower bounds on the corresponding chromatic numbers, Raigorodskii and Sagdeev~\cite{SagRai2019,Sag2019_FW} followed their argument (with some additional optimizations) and extracted the following quantitative result: $\chi(\R^n, \triangle) \ge \big(1.0140...+o(1)\big)^n$. However, as $k$ grows, the best lower bound on $\chi\big(\R^n, \triangle^k\big)$ that one can obtain via this technique decreases extremely rapidly\footnote{For instance, in~\cite{Sag2018_FR} Sagdeev only showed that $\chi(\R^n, \triangle^k) \ge \big(1+ 2^{-2^{k+4}}+o(1)\big)^n$.}.

Recently, Naslund~\cite{Naslund} used the slice-rank method (see~\cite{CLP,EllGij}) to give a very short argument leading to the following improvement of the lower bound on $\chi(\R^n, \triangle)$:
\begin{equation} \label{eq:chi_delta}
	\chi(\R^n, \triangle) \ge \big(1.0144...+o(1)\big)^n.
\end{equation}
It is not clear if one can get lower bounds on $\chi\big(\R^n, \triangle^k\big)$ for $k>2$ via similar technique.

Back to the classic results, in the paper~\cite{FrRod} Frankl and R\"odl found much more super-Ramsey configurations by proving the following `composition' statement.

\begin{Theorem} \label{th:prod}
	If $\mathcal{N}_1\subset \R^{d_1},\mathcal{N}_2 \subset \R^{d_2}$ are super-Ramsey, then so is $\mathcal{N}_1\times \mathcal{N}_2$.
\end{Theorem}

Frankl and R\"odl applied this result to show that the vertex set of any non-degenerate simplex is super-Ramsey. These two theorems, together with a simple observation that any subset $\M \subset \mathcal{N}$ of a super-Ramsey set $\mathcal{N}$ inherits this property\footnote{Indeed, any copy of $\mathcal{N}$ also contains a copy of $\M$. Thus $\alpha(V,\M) \le \alpha(V, \mathcal{N})$ for all $V \subset \R^n$. Similarly, we have $\chi(\R^n,\M) \ge \chi(\R^n,\mathcal{N})$ for all $n \in \N$.}, give all known examples of super-Ramsey (and even exponentially Ramsey) sets.

In particular, note that $\tetrahedron \subset \cube$ and $\triangle^k \subset \square^{k+1}$ for all $k \in \N$, where $\square^k$ is the set of all $2^k$ vertices of the $k$-dimensional hypercube with side length of $1/\sqrt{2}$ and $\cube$ is simply $\square^3$. Hence, the product theorem of Frankl and R\"odl, applied several times to a super-Ramsey two-point configuration, provides an alternative way to prove that $\triangle^k$ is super-Ramsey for all $k \in \N$. Sagdeev~\cite{Sag2018_CartProd} showed that a careful choice of the auxiliary parameters in this proof yields that
\begin{align}
	\chi(\R^n, \tetrahedron) & \ge \big(1.0136...+o(1)\big)^n, \label{eq:chi_delta3} \\
	\chi\big(\R^n, \triangle^k\big) & \ge \Big(1+\frac{1}{(k+1)^2\cdot2^{k+1}}+o(1)\Big)^n \mbox{ for all } k \ge 4. \label{eq:chi_deltak}
\end{align}

However, \Cref{th:prod} gives quantitatively inefficient bounds in the case when a set $\M$ can only be embedded into a very large Cartesian product $\mathcal{N}$, in particular, when $\M$ is a simplex. In this paper, we give a much more efficient way to concatenate constructions, using tree-like structures instead of product structures. This allows us to improve the lower bounds on the chromatic number of Euclidean space with any forbidden non-degenerate simplex. In particular, we get the following substantial improvements upon all the inequalities~\eqref{eq:chi_delta}, \eqref{eq:chi_delta3}, and~\eqref{eq:chi_deltak}. 

\begin{Theorem} \label{cor:simpl}
	For all $k \in \N$,  we have
	\begin{equation*}
		\chi\big(\R^n, \triangle^k\big) \ge \big(\psi_2^{1/(k+1)}+o(1)\big)^n
	\end{equation*}
	as $n \to \infty$, where $\psi_2$ is from the statement of \Cref{th:rai}. In particular, we have
	\begin{equation*}
		\chi(\R^n, \triangle) \ge \big(1.0742...+o(1)\big)^n, \mbox{ and } \ \chi(\R^n, \tetrahedron) \ge \big(1.0551...+o(1)\big)^n.
	\end{equation*}
\end{Theorem}

\vspace{1mm}

Observe that, as $k \to \infty$, the base of the exponent $\psi_2^{1/(k+1)} = 1+ \frac{0.214...}{k}+ O\big(\frac{1}{k^2}\big)$ tends to $1$ with roughly `linear' speed. So, the statement of \Cref{cor:simpl} is not only significantly stronger than the inequality~\eqref{eq:chi_deltak}, but also than the bound $\chi_{\mbox{\footnotesize \normalfont m}}\big(\R^n, \triangle^k\big) \ge \big(1+(3k-3)^{-2}+o(1)\big)^n$ on the \textit{measurable chromatic number} recently obtained for all $k\ge 3$ in~\cite{CSSV}.

We also provide explicit exponential lower bounds on the chromatic numbers of Euclidean space with forbidden right and acute triangles. Note that the case of an obtuse triangle requires some additional technical propositions from~\cite{FranklRodl1987}, and so we omit it in the present paper for simplicity.

\begin{Theorem} \label{cor:triangles}
	Let $\mathcal{R}$ and $\mathcal{A}$ be the sets of vertices of a right and an acute triangles, respectively. Then we have
	\begin{equation*}
		\chi(\R^n, \mathcal{R}) \ge \big(1.1133...+o(1)\big)^n, \mbox{ and } \ \chi(\R^n, \mathcal{A}) \ge \big(1.0742...+o(1)\big)^n.
	\end{equation*}
\end{Theorem}

Let us conclude this sections with a few remarks on the corresponding upper bounds. In case of the `classical' chromatic number $\chi(\R^n)$, the best result belongs to Larman and Rogers~\cite{Lar, Pros2020_LR} who showed that $\chi(\R^n) \le \big(3+o(1)\big)^n$. For all finite $\M \subset \R^d$, Prosanov~\cite{Pros2018_ExpRams} found an explicit $c(\M)\le 3$ such that $\chi(\R^n,\M) \le \big(c(\M)+o(1)\big)^n$. However, these upper bounds usually asymptotically much larger than the best current lower bounds. For instance, in case of an equilateral triangle it is only known that $\chi(\R^n, \triangle) \le \big(2.733+o(1)\big)^n$. Erd\H{o}s et al. showed that $\M$ is not Ramsey, i.e. that $\chi(\R^n,\M)$ is bounded from above by some constant independent of $n$, whenever $\M$ is non-spherical~\cite{EGMRSS1} or infinite~\cite{EGMRSS2}. Nevertheless, the classification of Ramsey configurations is a wide open problem, and there are even two different conjectures upon the answer, see~\cite{Leader}. Finally, note that it is unknown if there is a Ramsey set that is not exponentially (or super-) Ramsey.

\subsection{Weak sunflowers} \label{sec1:sun}

A collection of $k\ge 3$ sets is called a \textit{$k$-sunflower} if all their pairwise intersections coincide. Erd\H{o}s and Rado~\cite{ErdRad1960} introduced this notion, proved that any family $\F$ of $r$-element sets with no $k$-sunflower satisfies $|\F| \le r!(k-1)^r$, and conjectured that this bound can be further improved to $|\F| \le C^r$ for some $C=C(k)$ depending only on $k$. Their conjecture remains one of the most famous problems in modern Combinatorics and is wide open even in case $k=3$. Recently, there was a breakthrough result due to  Alweiss, Lovett, Wu and Zhang~\cite{ALWZ} that significantly improved the upper bound.

Various analogues, generalizations, and related notions have been extensively studied over the past decades including the following one. A collection of $k\ge 3$ sets is called a \textit{weak $k$-sunflower} if all their pairwise intersections are of the same cardinality. Let $G_k(n)$ stands for the maximum size of a family $\F \subset 2^{[n]}$ that contains no weak $k$-sunflowers. Kostochka and R\H{o}dl~\cite{KosRodl1998} proved that
\begin{equation} \label{eq:snflwr1}
	G_k(n)\ge k^{c(n\log n)^{1/3}}
\end{equation}
for some absolute constant $c>0$, and their result remains the best known.

Concerning the upper bounds, there was a conjecture of Erd\H{o}s and Szemer\'edi~\cite{ErdSzem1978} stating that $G_3(n) \le (2-\varepsilon)^n$ for a sufficiently small positive $\varepsilon$. It was also proved by  Frankl and R\"odl~\cite{FranklRodl1987}. In fact, they proved that for any $k \ge 3$ there exists a sufficiently small $\varepsilon_k>0$ such that $G_k(n) \le \big(2-\varepsilon_k+o(1)\big)^n$ asymptotically. Frankl and R\"odl did not provide explicit estimates for these $\varepsilon_k$ but, as $k$ grows, one cannot hope to achieve values larger than $ 2^{-2^{k+o(k)}}$ using their technique, see~\cite{Sag2018_FR}, and even in case $k=3$ the optimal value of $\varepsilon_3$ would be smaller than $0.01$.

Naslund~\cite{Naslund22} substantially improved upon this bound for $k=3$ by showing that
\begin{equation} \label{eq:snflwr_nas}
	G_3(n) \le \big(1.837+o(1)\big)^n.
\end{equation}
His proof was based on a relatively recent breakthrough of Ellenberg and Gijswijt~\cite{EllGij} in the capset problem via the slice-rank method. See also the prior paper~\cite{NasSaw2017} of Naslund and Sawin on the sunflower problem and the original celebrated result of Croot, Lev, and Pach~\cite{CLP} for more examples of using the slice-rank method in Combinatorics. Nevertheless, it is unknown if one can apply this argument to get a non-trivial upper bound on $G_k(n)$ for $k>3$.

In the present paper, we show that our simple ideas from \Cref{sec2:tool} combined with the classic theorem by Frankl and Wilson~\cite{FranklWil1981}  lead to the following exponential upper bound on $G_k(n)$.

\begin{Theorem} \label{cor:snflwr}
	For all $k \ge 3$, if $\F \subset 2^{[n]}$ contains no weak $k$-sunflowers, then
	\begin{equation*}
		|\F| \le \big(2\psi^{-1/k}+o(1)\big)^{n}
	\end{equation*}
	as $n \to \infty$, where $\psi = \frac{1+\sqrt{2}}{2} = 1.207...$
\end{Theorem}

Observe that in case $k=3$, this theorem only implies that $G_3(n) \le \big(1.879+o(1)\big)^n$, while \eqref{eq:snflwr_nas} provides a better estimate. Nevertheless, as $k \to \infty$, the value $2\psi^{-1/k} = 2-\frac{0.376...}{k}+ O\big(\frac{1}{k^2}\big)$ tends to $2$ with roughly `linear' speed, and greatly improves upon the bounds following form the prior techniques for all $k \ge 4$.

\subsection{Frankl--R\"odl from Frankl--Wilson} \label{sec1:FR}

Given $r<n$, let $\F \subset \binom{[n]}{r}$ be a family of $r$-element subsets of $[n]$. This family is said to be {\it $s$-avoiding} if $|F_1\cap F_2| \neq s$ for all $F_1,F_2 \in \F$. It is easy to check that $\F \subset \binom{[n]}{r}$ is $s$-avoiding if and only if $\overline{\F}\coloneqq\big\{ \overline{F}: F \in \F\big\} \subset \binom{[n]}{n-r}$ is $(n-2r+s)$-avoiding.  Hence, we can assume without loss of generality that $r \le n/2$.
In the present section, we assume that both $r$ and $s$ grow linearly with $n$, i.e. that $r\sim\rho n$ and $s\sim\sigma n$ for some $0<\sigma<\rho\le 1/2$ as $n \to \infty$. Under these conditions, Frankl and Wilson~\cite{FranklWil1981} applied polynomial method to show that if $r-s$ is prime or a prime power, then the size of any $s$-avoiding family is exponentially smaller than $\binom{n}{r}$. Their upper bound is known to be asymptotically tight if $\sigma\le\rho/2$, see \cite{BobKupRai2016}. However, in case $\sigma>\rho/2$, currently the best upper bound is due to Ponomarenko and Raigorodskii~\cite{PonRai2013}, but its tightness is indefinite. We summarize these bounds in the following statement. (The asymptotic equivalence between the original technical result of Ponomarenko and Raigorodskii and its simpler form that we use here was shown in~\cite{RaiSag2018}.)

%\textcolor{red}{This theorem also looks kind of ugly, I'm not sure we need it to take so much space.}
\begin{Theorem} \label{th:FWPR}
	Given positive $\sigma<\rho\le 1/2$, assume that $r=r(n)\sim \rho n$, $s=s(n)\sim\sigma n$ as $n \to \infty$, and that $r-s$ is prime or a prime power for all $n$. Then any $s$-avoiding family $\F \subset \binom{[n]}{r}$ satisfies
	\begin{equation*}
		|\F| \le \binom{n}{r}\big(\delta(\rho,\sigma)+o(1)\big)^n,
	\end{equation*}
	where
	\begin{equation*}
		\delta(\rho,\sigma) =
		\begin{cases}
			\exp\big\{H(\rho-\sigma)-H(\rho)\big\} & \mbox{if } \sigma < \rho/2; \\
			\exp\big\{H(\rho-\sigma)-H(2\rho-2\sigma)\big\} & \mbox{if } \sigma \ge \rho/2,
		\end{cases}
	\end{equation*}
	and $H(x) = -x\ln(x)-(1-x)\ln(1-x)$ is the entropy function.
\end{Theorem}

Frankl and R\"odl proved in their seminal paper~\cite{FranklRodl1987} that the size of any $s$-avoiding family is still  exponentially smaller than $\binom{n}{r}$ even if $r-s$ is neither prime nor a prime power. However, it is rather a laborious task to extract some quantitative bounds from their proof (which they prudently omitted). Keevash and Long~\cite{KeevLong2017} have recently showed that this general statement in fact follows from the case of prime $r-s$ settled by Frankl and Wilson~\cite{FranklWil1981} combined with the idea that any natural number can be represented as a sum of at most $4$ primes. We show that our simple ideas from~\Cref{sec2:tool} combined with the ideas of Keevash and Long give a shorter and more efficient reduction.

\begin{Theorem} \label{cor:FR}
	Given positive $\sigma<\rho\le 1/2$, assume that $r=r(n)\sim \rho n$, $s=s(n)\sim\sigma n$ as $n \to \infty$. Put $c=3$ if $r-s$ is odd, and $c=4$ otherwise. Then any $s$-avoiding family $\F \subset \binom{[n]}{r}$ satisfies
	\begin{equation*}
		|\F| \le \binom{n}{r}\big(\delta(\rho,\sigma)+o(1)\big)^{n/c},
	\end{equation*}
	where $\delta(\rho,\sigma)$ is from the statements of \Cref{th:FWPR}.
\end{Theorem}

It might be possible to replace $c=4$ in the last exponent with a better $c=2$ in case $r-s$ is even, see the discussion in \Cref{sec3:FR}.

Let us briefly numerically compare this bound with prior results in the special `centered' case $r\sim n/2, s\sim n/4$. If $r-s$ is prime or a prime power, then \Cref{th:FWPR} implies that any $s$-avoiding family $\F \subset \binom{[n]}{r}$ satisfies $|\F| \le \big(1.755+o(1)\big)^n$ and this is tight. If $r-s$ is only assumed to be odd, then following the proof from~\cite{KeevLong2017}, one can only deduce that $|\F| \le \big(1.998+o(1)\big)^n$, while Frankl and R\"odl themselves~\cite{FranklRodl1987} showed that $|\F| \le \big(1.99+o(1)\big)^n$ in this case, and our \Cref{cor:FR} yields that $|\F| \le \big(1.915+o(1)\big)^n$. Similar quantitative improvements hold for other values of $\rho,\sigma$ as well.

One may consider \Cref{cor:FR} as a tool to guarantee an `edge', i.e. a pair of sets with a prescribed cardinality of their intersection, in every sufficiently large uniform family. Under some additions constraints on the parameters, our technique can even guarantee a `clique'.

\begin{Theorem} \label{cor:clique}
	Given positive $\sigma<\rho\le 1/2$, assume that $r=r(n)\sim \rho n$, $s=s(n)\sim\sigma n$ as $n \to \infty$. Put $c=1$ if $r-s$ is prime or a prime power, $c=3$ if $r-s$ is another odd number, and $c=4$ if $r-s$ is even. Then for all $k \ge 3$, any family $\F \subset \binom{[kn]}{kr}$ that contains no collection of $k$ sets with pairwise intersections of cardinality $2s+(k-2)r$ satisfies 
	\begin{equation*}
		|\F| \le \binom{kn}{kr}\big(\delta(\rho,\sigma)+o(1)\big)^{n/c},
	\end{equation*}
	where $\delta(\rho,\sigma)$ is from the statements of \Cref{th:FWPR}.
\end{Theorem}

\subsection{Manhattan Ramsey Theory} \label{sec1:MRT}

Let $\R_1^n$ be an $n$-dimensional space equipped with the Manhattan norm\footnote{Recall that the Manhattan $\ell_1$-norm is defined by $\|\x\|_1 \coloneqq \sum_{i=1}^n|x_i|$ for all $\x=(x_1,\dots,x_n) \in \R^n$.}. One can easily translate the range of Euclidean problems considered in \Cref{sec1:ERT} to Manhattan space. Namely, for a subset $\M \subset \R^d$, the {\it chromatic number $\chi(\R_1^n,\M)$} is defined as the smallest $r$ such that there exists an $r$-coloring of $\R^n$ with no monochromatic $\ell_1$-isometric copy\footnote{A subset $\M' \subset \R^n$ is called an {\it $\ell_1$-isometric copy} of $\M$ if there exists a bijection $f: \M \to \M'$ such that $\|\x-\y\|_1 = \|f(\x)-f(\y)\|_1$ for all $\x,\y \in \M$.} of $\M$. For a subset $V \subset \R^n$, the \textit{independence number} $\alpha(V,\M)$ stands for the maximum cardinality $|W|$ of a subset $W\!\subset\! V$ with no $\ell_1$-isometric copy of $\M$. Finally, we call $\M\subset \R^d$ {\it super-Ramsey}, if for a sufficiently small $\varepsilon>0$ and for any dimension $n$, there is a set $V\subset \R^n$ of size that is at most exponential in $n$ such that $|V|/\alpha(V, \M) \ge \big(1+\varepsilon+o(1)\big)^n$. 

In the present paper, we show that the following general statement almost trivially follows from \Cref{th:prod} and the classic result of Frankl and Wilson~\cite{FranklWil1981}, although it has never been stated explicitly to our knowledge.

\begin{Theorem} \label{th:fr_manh}
	For all $d \in \N$, any finite $\M \subset \R^d$ is super-Ramsey in case of the Manhattan norm. In particular, there exists $\varepsilon=\varepsilon(\M)>0$ such that
	\begin{equation*}
		\chi(\R_1^n,\M) \ge \big(1+\varepsilon+o(1)\big)^n.
	\end{equation*}
\end{Theorem}

We also use our tree-like structures to provide relatively strong explicit exponential lower bounds on these chromatic numbers in some important special cases, see the exact statements and their proofs in \Cref{sec3:MRT}.

\section{Tools} \label{sec2:prel}

\subsection{Tree-like concatenation} \label{sec2:tool}

Our main new trick to combine the constructions consists of two elements. First is a simple graph-theoretic lemma. Recall that, for a graph $G = (V,E)$, its independence number $\alpha(G)$ is the size of the largest set of vertices that induces no edges.

\begin{Lemma} \label{lemgt}
	Assume that $G_i = (V, E_i)$, $i\in [k]$ are graphs on the same vertex set. Fix a tree $T$ on $k$ arbitrarily ordered edges. Then in any subset $W\subset V$ of size strictly bigger than $\alpha(G_1)+\ldots+\alpha(G_k)$ there exists a homomorphic copy $h(T)$ of $T$ (that is, not necessarily with all vertices distinct) in $G = (V, E_1\cup\ldots\cup E_k)$ such that the image of the $i$-th edge of $T$ in $h(T)$ belongs to $E_i$.
\end{Lemma}
\begin{proof}
    
The proof of this statement is a simple induction on $k$. For $k=1$ the statement is obvious because the only tree on $2$ vertices is an edge, and the collection of vertices that contains no homomorphic copies of an edge must form an independent set.
%violate the aforementioned property must form an independent set.

To do the induction step from $k-1$ to $k$, take any tree $T$ with $k$ edges and a non-leaf vertex $v\in T$. Then, split the tree $T$ into two non-empty subtrees $T',T''$ arbitrarily such that $v$ is their only common vertex. Clearly, both $T'$ and $T''$ have at most $k-1$ edges. Let $U\subset [k]$ be the set of indices of edges in $T'$. (Then $[k]\!\setminus\! U$ is the set of indices of edges in $T''$.) 

By induction, there are at most $\sum_{i\in U} \alpha(G_i)$ vertices $w\in W$ such that no homomorphic copy of $T'$ in $W$ satisfies $h(v) = w$. Indeed, the set of all such $w\in W$ contains no copies of $T'$ since no vertex can play the role of $v$ in the image of $T'$. Similarly, there are at most $\sum_{i\in [k]\setminus U} \alpha(G_i)$ vertices $w\in W$ such that no homomorphic copy of $T''$ in $W$ satisfies $h(v) = w$.

If $|W|>\sum_{i\in[k]}\alpha(G_i)$, then there is a vertex $w\in W$ not falling in either of these two categories, and we get the desired copy of $T$ with $h(v) = w$.
\end{proof}

The other important ingredient for us is orthogonality. It allows to turn seemingly `non-rigid' structures like paths or stars into `rigid' cliques. In graph-theoretic terms, we need to work with products of graphs. Below, we introduce it in the form that would be convenient for applications.

Let $G_i=(V_i,E_i)$, $i \in [k]$, be a family of $k$ graphs. For any choice or their edges $(u_i,w_i) \in E_i,$ $i \in [k]$, we call the sequence
\begin{equation*}
	\W_0 \coloneqq (w_1,\dots,w_k), \ \W_1 \coloneqq(u_1,w_2,\dots,w_k), \ \dots \ , \  \W_k \coloneqq(w_1,\dots,w_{k-1},u_k).
\end{equation*}
an {\it orthogonal star} and the sequence
\begin{equation*}
	\U_0 \coloneqq (w_1,\dots,w_k), \ \U_1 \coloneqq (u_1,w_2,\dots,w_k), \ \U_2 \coloneqq (u_1,u_2,w_3,\dots,w_k), \ \dots \ , \ \U_k \coloneqq (u_1,\dots,u_k)
\end{equation*}
an {\it orthogonal path} in the Cartesian product $\VV$.

%The following lemma provides a simple upper bound on the maximum cardinality of a subset $W \subset \VV$ that does not contain either orthogonal stars or paths.

\begin{Lemma} \label{lem:star}
	In the above notation, if $W \subset \VV$ contains no orthogonal stars (or no orthogonal paths), then
	\begin{equation*}
		\frac{|W|}{|\VV|} \leq \sum_{i=1}^{k} \frac{\alpha(G_i)}{|V_i|}.
	\end{equation*}
\end{Lemma}

\begin{proof}
For all $i \in [k]$, consider a graph $G_i':=(\VV,E_i')$, where an edge connects two $k$-tuples $(u_1,\dots,u_k)$ and $(w_1,\dots,w_k)$ if and only if $(u_i,w_i) \in E_i$ and $u_j=w_j$ for all $j \neq i$.

%Put $G_i':=(V_1\times\ldots\times V_k,E_i')$, where 
%$$E_i'= \{\{\W,\U\}: w_j=u_j \text{ for } j\ne i \text{ and } \{w_i,u_i\}\in E_i\}.$$
The first thing to note is that $\frac{\alpha(G_i)}{|V_i|}= \frac{\alpha(G_i')}{|\VV|}$. Then $G'_i$ all live on the same set of vertices, and we only need to apply \Cref{lemgt} with $T$ being a star (or a path). It is straightforward to see that any homomorphic copy of such $T$ must be an orthogonal star (or an orthogonal path).
\end{proof}

\vspace{1mm}

Finally, we conclude with a straightforward technical proposition that we will need.
\begin{Proposition}\label{lem:part}
	Given $r<n$, let $\F \subset \binom{[n]}{r}$ be a family of $r$-element subsets of $[n]$. Then for all partitions of $n$ and $r$ into $k$ non-negative summands $n=n_1+\dots+n_k$ and $r=r_1+\dots+r_k$, there is a partition of $[n]$ into $k$ disjoint subsets $[n]=N_1\cup\dots\cup N_k$ satisfying the following two properties. First, $|N_i|=n_i$ for all $i \in [k]$. Second, the size of the subfamily
	\begin{equation*}
		\F' \coloneqq \big\{F \in \F: |F\cap N_i| = r_i \mbox{ for all } i \in [k]\big\}
	\end{equation*}
	satisfies
	\begin{equation*}
		|\F'|\cdot \binom{n}{r} \ge |\F|\cdot\binom{n_1}{r_1}\cdots\binom{n_k}{r_k}.
	\end{equation*}
\end{Proposition}
\begin{proof} %Take a random partition of $[n]$ into parts of desired size. For any given set from ${[n]\choose r}$ the probability that it `respects' this partition is the same, and thus with probaThen, for any set that intersects the
	Let us consider all $\frac{n!}{n_1!\cdots n_k!}$ possible partitions of $[n]$ into $k$ disjoint subsets having the first desired property. It is easy to see that each $F \in \F$ contributes to the sizes of exactly
	\begin{equation*}
		\frac{r!}{r_1!\cdots r_k!} \cdot \frac{(n-r)!}{(n_1-r_1)!\cdots (n_k-r_k)!}
	\end{equation*}
	corresponding subfamilies. Now the pigeonhole principle finishes the proof.
\end{proof}

\subsection{Frankl--R\"odl product theorem} \label{sec2:FR}

In the present section we give a proof of \Cref{th:prod} from~\cite{FrRod} with minor exposition modifications, allowing the reader to compare our tree-like construction from \Cref{lem:star} with its precursor.

So, assume that both $\mathcal{N}_1\subset \R^{d_1}$ and $\mathcal{N}_2 \subset \R^{d_2}$ are super-Ramsey. Namely, that there exist constants $c_i, \varepsilon_i>0$ and sets $V_i(n) \subset \R^n$ such that $|V_i(n)| \le c_i^n$ and $|V_i(n)|/\alpha(V_i(n), \mathcal{N}_i) \ge (1+\varepsilon_i)^n$ for all $n \in \N$, $i=1,2$.

For large $n$, we consider a partition $n=n_1+n_2$, where $n_1\sim(1-\eta) n, n_2 \sim \eta n$, and the value of $\eta$ will be specified later. Put $V_1 = V_1(n_1)$, $V_2 = V_2(n_2)$, and $V = V_1\times V_2 \subset \R^n$. We will show that  the Cartesian product $\mathcal{N}_1\times\mathcal{N}_2$ is also super-Ramsey using this set. First, it is clear that $|V| \le c_1^{n_1}c_2^{n_2}$, i.e. that the size of $V$ is at most exponential with $n$. In what follows, we show that the size of any subset $W \subset V$ with no copies of $\mathcal{N}_1\times\mathcal{N}_2$ is exponentially smaller than $|V|$.

For all $\x \in V_1$, let $w_{\x}$ stand for the number of copies of $\mathcal{N}_2$ in the `layer' $W_{\x} \coloneqq W\cap \big(\{\x\}\times V_2\big)$. Similarly, for a copy $Y\subset V_2$ of $\mathcal{N}_2$, let $w_{Y}$ stand for the number of layers, indexed by $\x$, such that $\{\x\}\times Y \subset W_{\x}$. It is clear that $\sum_{\x} w_{\x} = \sum_{Y} w_{Y}$. To give an upper bound on $|W|$, we now estimate these sums.

On the one hand, each $w_{Y}$ does not exceed $\alpha(V_1, \mathcal{N}_1)$ since otherwise we will find a copy of $\mathcal{N}_1\times\mathcal{N}_2$ in $W$. Moreover, note that there are at most $|V_2|^m$ ways to choose $Y$, where $m=|\mathcal{N}_2|$. Hence,
\begin{equation} \label{eq:prod1}
	\sum_{Y \subset V_2} w_{Y} \le \alpha(V_1, \mathcal{N}_1) \cdot |V_2|^m \le  \frac{|V_1|\cdot|V_2|^m}{(1+\varepsilon_1)^{n_1}} \le |V|\cdot\frac{c_2^{(m-1)n_2}}{(1+\varepsilon_1)^{n_1}}.
\end{equation}

On the other hand, observe that each layer $W_{\x}$ contains at least $|W_{\x}| - \alpha(V_2, \mathcal{N}_2)$ copies of $\mathcal{N}_2$. To see this, while $|W_{\x}| > \alpha(V_2, \mathcal{N}_2)$, discard elements of the layer one by one, reducing the number of copies of $\mathcal{N}_2$ by at least $1$ in the remainder on each step\footnote{Here Frankl and R\"odl used a `weaker' bound $w_{\x} \ge \lfloor|W_{\x}|/\alpha(V_2, \mathcal{N}_2)\rfloor$, while one can  get a `stronger' estimate than ours by applying a  probabilistic argument similar to the one  used in the proof of the \textit{crossing lemma}, see~\cite{AlonSpencer}. However, these modifications would not change the resulting exponent asymptotically.}. Therefore,
\begin{equation} \label{eq:prod2}
	\sum_{\x \in V_1} w_{\x} \ge \sum_{\x \in V_1} |W_{\x}| - |V_1|\cdot\alpha(V_2, \mathcal{N}_2) \ge |W| - \frac{|V|}{(1+\varepsilon_2)^{n_2}}.
\end{equation}

Finally, by combining the inequalities \eqref{eq:prod1} and \eqref{eq:prod2}, we conclude that
\begin{equation} \label{eq:prod3}
	\frac{|W|}{|V|} \le \frac{1}{(1+\varepsilon_2)^{n_2}} + \frac{c_2^{(m-1)n_2}}{(1+\varepsilon_1)^{n_1}}.
\end{equation}
Now it remains only to note that the right-hand side decreases exponentially with $n$ if the value of $\eta$ is sufficiently small. More specifically, one can easily check that, by taking $\eta = \frac{\log(1+\varepsilon_1)}{\log\big((1+\varepsilon_1) (1+\varepsilon_2) c_2^{m-1}\big)}$, we get that $|W|\le |V|/\big(1+\varepsilon_2+o(1)\big)^{\eta n}$, which is asymptotically an optimal choice here. 

\section{Applications}\label{sec3}

\subsection{Euclidean Ramsey theory} \label{sec3:ERT}

For all $k \in \N$,  we call the subset\footnote{Note that here and in what follows we do not distinguish points and their position vectors.} $\{\boldsymbol{0},\e_1,\dots,\e_k\} \subset \R^k$, where $\e_i$ stands for the $i$'th standard basis vector in $\R^k$, a {\it $k$-semicross} and denote it by $\SC^k$. Given $V \subset \R^n$, consider a unit distance graph $G(V)=(V,E(V))$, where two points of $V$ are connected if and only if they are at unit distance apart. It is easy to see that any orthogonal star in the Cartesian power $V^k$ is isometric to $\SC^k$. Hence, \Cref{lem:star} implies the following upper bound\footnote{Observe that two independence numbers, $\alpha(G(V))$ and $\alpha(V)$ defined in the introduction, are identical.} on the maximum cardinality of a subset of $V^k \subset \R^{nk}$ with no isometric copies of $\SC^k$:
\begin{equation} \label{eq:alpha_SC}
	\alpha\big(V^k, \SC^k\big) \le k|V|^{k-1}\alpha(V).
\end{equation}

Now we apply the last inequality to the set $V(n) \subset \R^n$ from the statements of \Cref{th:rai} and combine the result with \eqref{eq:chi_alpha2}. This shows that
\begin{equation} \label{eq:chi_SK}
	\chi\big(\R^{kn}, \SC^k\big) \ge \frac{|V(n)|^k}{\alpha\big(V(n)^k, \SC^k\big)} \ge \frac{|V(n)|^k}{k|V(n)|^{k-1}\alpha(V(n))} = \frac{|V(n)|}{k\alpha(V(n))} \ge \frac{\big(\psi_2+o(1)\big)^n}{k}.
\end{equation}
If $k \in \N$ is fixed and $n \to \infty$, \eqref{eq:chi_SK} implies that
\begin{equation*} \label{eq:chi_SK2}
	\chi\big(\R^{n}, \SC^k\big) \ge \big(\psi_2^{1/k}+o(1)\big)^n.
\end{equation*}

Finally, a simple observation that $(k+1)$-semicross contains $k+1$ vertices of the regular $k$-dimensional simplex with side length of $\sqrt{2}$ completes the proof of \Cref{cor:simpl}.

\vspace{3mm}

Let us consider the following `asymmetric' generalization of the above argument. Given $k$ positive reals $\lambda_1,\dots,\lambda_k$, we call the subset $\{\boldsymbol{0},\lambda_1\e_1,\dots,\lambda_k\e_k\} \subset \R^k$ a {\it scaled $k$-semicross} and denote it by 	$\SC^k(\lambda_1,\dots,\lambda_k)$. One can easily check that any orthogonal star in the scaled Cartesian product $(\lambda_1V)\times\dots\times(\lambda_kV)$ is an isometric copy of $\SC^k(\lambda_1,\dots,\lambda_k)$. Thus, \Cref{lem:star} yields the following generalization of \eqref{eq:alpha_SC}:
\begin{equation*}
	\alpha\big((\lambda_1V)\times\dots\times(\lambda_kV), \SC^k(\lambda_1,\dots,\lambda_k)\big) \le k|V|^{k-1}\alpha(V).
\end{equation*}
As earlier, for fixed values of $k$, the last inequality implies that
\begin{equation*}
	\chi\big(\R^n, \SC^k(\lambda_1,\dots,\lambda_k)\big) \ge \big(\psi_2^{1/k}+o(1)\big)^n.
\end{equation*}

Observe that a right triangle with catheti lengths  $a$ and $b$ is isometric to the scaled $2$-semicross $\SC^2(a,b)$, while an acute triangle with side lengths of $a,b,$ and $c$ can be embedded into the scaled $3$-semicross
\begin{equation*}
\SC^3\big(\sqrt{(a^2+b^2-c^2)/2}, \sqrt{(b^2+c^2-a^2)/2}, \sqrt{(c^2+a^2-b^2)/2}\,\big).
\end{equation*}
This completes the proof of \Cref{cor:triangles}.

\vspace{2mm}

We remark that the argument from this section can be directly applied to the following problem considered by Naslund in~\cite{Naslund23}: provide a lower bound on the number of colors needed to color $\R^n$ such that no simplex with side lengths from a fixed $m$-element set is monochromatic. More specifically, one can use the values from the first column of the table at the end of~\cite[Section~1.2]{Naslund23} in the same way we use \Cref{th:rai} in this paper to numerically improve upon some of the values from the latter columns.

\subsection{Manhattan Ramsey theory} \label{sec3:MRT}

Though Frankl and Wilson~\cite{FranklWil1981} considered only Euclidean norm, their technique similarly works for all $\ell_p$-spaces. In particular, it can be argued that they proved a pair of points at unit distance apart to be a super-Ramsey configuration in case of Manhattan norm as well. The current quantitative record here is via the following analogue of \Cref{th:rai} (also by Raigorodskii~\cite{Rai4}).

\begin{Theorem} \label{th:rai2}
	There is a sequence of subsets $V(n) \subset \R_1^n$, $n \in \N$, such that
	$ \frac{|V(n)|}{\alpha(V(n))} \ge \big(\psi_1+o(1)\big)^n,$
	where $\psi_1 = \frac{1+\sqrt{3}}{2} = 1.366...$
\end{Theorem}

Similarly, one can easily observe that the proof of \Cref{th:prod} presented in \Cref{sec2:FR} holds not only for Euclidean but also for Manhattan space and, more generally, for all $\ell_p$-spaces. Hence, when applied several times to a super-Ramsey two-point configuration, \Cref{th:prod} implies that the vertex set of any hyperrectangle, or a \textit{box}, is super-Ramsey.

For all positive reals $\lambda_1,\dots,\lambda_k$, we call the one-dimensional set $\{0, \lambda_1, \lambda_1+\lambda_2, \dots, \sum_{i=1}^k \lambda_i\}$ a {\it baton} and denote by $\B(\lambda_1,\dots,\lambda_k)$ for a shorthand. Observe that $\B(\lambda_1,\dots,\lambda_k)$ can be $\ell_1$-isometrically embedded into a $k$-dimensional box with side lengths of $\lambda_1,\dots,\lambda_k$. Therefore, all batons are also super-Ramsey. This statement is crucial for Manhattan Ramsey theory due to the following simple observation: each finite set $\M = \{\x^0,\ldots,\x^k\}\subset \R^d$ is a subset of the $d$-dimensional grid $\prod_{i=1}^d\{x^0_i,\ldots, x^k_i\}$, i.e. a Cartesian product of some batons. Now we apply \Cref{th:prod} again to this product and finish the proof of \Cref{th:fr_manh}.

\vspace{3mm}

As we discussed in~\Cref{sec1:ERT}, direct embedding of a simplex or a baton into a box leads to an exponentially small (in terms of the box dimension) value of the corresponding $\varepsilon$ from the statement of \Cref{th:fr_manh}, which is rather a pour bound. Using our tree-like structures instead, we provide the following strong quantitative improvements.

\begin{Theorem} \label{cor:manh_sympl}
	For all $k \in \N$ and all positive reals $\lambda_1,\dots,\lambda_k$, we have
	\begin{equation*}
		\chi\big(\R_1^n, \B(\lambda_1,\dots,\lambda_k)\big) \ge \big(\psi_1^{1/k}+o(1)\big)^n, \mbox{ and } \ \chi\big(\R_1^n, \triangle^k\big) \ge \big(\psi_1^{1/(k+1)}+o(1)\big)^n
	\end{equation*}
	as $n \to \infty$, where $\psi_1$ is from the statement of \Cref{th:rai2}. In particular, we have
	\begin{equation*}
		\chi(\R_1^n, \triangle) \ge \big(1.1095...+o(1)\big)^n, \mbox{ and } \ \chi(\R_1^n, \tetrahedron) \ge \big(1.0810...+o(1)\big)^n.
	\end{equation*}
	In fact, the penultimate bound holds not only for regular but also for all triangles.
\end{Theorem}

The proof of this statement basically repeats the ideas from \Cref{sec3:ERT}, so we stress only the minor distinctions below. Given $V \subset \R^n$, consider a unit distance graph $G(V)=(V,E(V))$, where two points of $V$ are connected if and only if they are at unit \textit{Manhattan} distance apart. It is easy to see that any orthogonal star in the scaled Cartesian product $(\lambda_1V)\times\dots\times(\lambda_kV)$ is $\ell_1$-isometric to the scaled $k$-semicross $\SC^k(\lambda_1,\dots,\lambda_k)$, while any orthogonal path there is $\ell_1$-isometric to the baton $\B(\lambda_1,\dots,\lambda_k)$. Hence, as in the previous section, we can apply \Cref{lem:star} to the sets $V(n)$ from the statement of \Cref{th:rai2} to conclude that if $\M$ is either a baton $\B(\lambda_1,\dots,\lambda_k)$ or a scaled $k$-semicross $\SC^k(\lambda_1,\dots,\lambda_k)$, then
\begin{equation*}
	\chi(\R_1^n, \M) \ge \big(\psi_1^{1/k}+o(1)\big)^n
\end{equation*}
as $n \to \infty$, where $\psi_1$ is from the statement of \Cref{th:rai2}. To finish the proof of \Cref{cor:manh_sympl} it remains only to note that regular $k$-dimensional simplex can be $\ell_1$-isometrically embedded into a $(k+1)$-semicross, while any triangle with side lengths of $a,b,$ and $c$ can be $\ell_1$-isometrically embedded into the scaled $3$-semicross $\SC^3\big((a+b-c)/2, (b+c-a)/2, (c+a-b)/2\big)$.

\vspace{3mm}

At the end of this section, we mention that, for all $\ell_p$-spaces, one could  get similar lower bounds $\chi\big(\R_p^n, \triangle^k\big) \ge \big(\psi^{1/(k+1)}+o(1)\big)^n$ with $\psi=\frac{1+\sqrt{2}}{2} = 1.207...$  using the original result of Frankl and Wilson~\cite{FranklWil1981}. In case $p=\infty$ much better bounds are known, see~\cite{FKS}.

\subsection{Frankl--R\"odl from Frankl--Wilson} \label{sec3:FR}

We begin the proof of \Cref{cor:FR} with considering the case when $r-s$ is odd. Put $n_1 = n_2 = \lfloor n/3\rfloor$, $n_3=n-n_1-n_2$ and $r_1 = r_2 = \lfloor r/3\rfloor$, $r_3=r-r_1-r_2$. Let us consider a partition $[n]=N_1\cup N_2 \cup N_3$ of $[n]$ into $3$ disjoint parts of sizes $n_1, n_2,$ and $n_3$, respectively, guaranteed by \Cref{lem:part}, i.e. such that the
subfamily
\begin{equation*}
	\F' \coloneqq \big\{F \in \F: |F\cap N_i| = r_i \mbox{ for all } i = 1,2,3\big\}
\end{equation*}
satisfies
\begin{equation} \label{eq:FR0}
	|\F'|\cdot \binom{n}{r} \ge |\F|\cdot\binom{n_1}{r_1}\binom{n_2}{r_2}\binom{n_3}{r_3}.
\end{equation}

Several strengthenings of the ternary Goldbach problem are known, ensuring that every odd number can be expressed as the sum of three `almost equal' primes. One of the best bounds on the error term is due to Matom\"aki, Maynard, and Shao~\cite{MatMayShao2017}, who showed that every odd $x>5$ can be expressed as a sum of three primes, each having the form $x/3+o(x^{0.551})$. In particular, we can write $r-s$ as the sum of three primes $r-s=p_1+p_2+p_3$ with $p_i \sim (r-s)/3$ for all $i=1,2,3$. Put $s_i = r_i-p_i$. It is easy to see that $s_1+s_2+s_3=s$ and that $s_i\sim s/3$ asymptotically for all $i=1,2,3$.

Given $i \in \{1,2,3\}$, let us consider the following graph $G_i=(V_i,E_i)$. Its set of vertices $V_i = \binom{N_i}{r_i}$ consists of all $r_i$-element subsets of $N_i$. An edge connects two vertices whenever they share exactly $s_i$ common elements. It is clear that the independent sets are exactly the $s_i$-avoiding families by the definition. Recall that $|N_i|=n_i, r_i \sim \rho n_i, s_i \sim \sigma n_i$, and that $r_i-s_i=p_i$ is prime. Hence, we can apply \Cref{th:FWPR} to show that
\begin{equation} \label{eq:FR2}
	\alpha(G_i) \le \binom{n_i}{r_i}\big(\delta(\rho,\sigma)+o(1)\big)^{n_i}.
\end{equation}

Observe that for all $F \in \F'$, we have $F\cap N_i \in \binom{N_i}{r_i}$ by the definition. Therefore, we can think of $\F'$ as of a subset of the Cartesian product $V_1\times V_2 \times V_3$, so that $F \in \F'$ corresponds to the triple $(F\cap N_1, F\cap N_2, F\cap N_3)$. Moreover, observe that any two sets $F_0, F_3 \in \F'$, corresponding to the endpoints $\U_0 = (w_1,w_2,w_3)$ and $\U_3=(u_1,u_2,u_3)$ of an orthogonal path in $V_1\times V_2 \times V_3$, respectively, share exactly
\begin{equation*}
	|w_1\cap u_1|+|w_2\cap u_2|+|w_3\cap u_3|=s_1+s_2+s_3=s
\end{equation*}
common elements. Being a subfamily of $\F$, the family $\F'$ is also $s$-avoiding, and we conclude that there are no orthogonal paths in $\F'$. Thus, \Cref{lem:star} combined with~\eqref{eq:FR2} yields that
\begin{equation*}
	\frac{|\F'|}{\binom{n_1}{r_1}\binom{n_2}{r_2}\binom{n_3}{r_3}} \le \sum_{i=1}^{3} \big(\delta(\rho,\sigma)+o(1)\big)^{n_i} = \big(\delta(\rho,\sigma)+o(1)\big)^{n/3}.
\end{equation*}
Now we use~\eqref{eq:FR0} to get the desired upper bound on $|\F|$. This completes the proof of \Cref{cor:FR} in case when $r-s$ is odd.

We can argue similarly for even $r-s$ as well. Observe that if it were known that every sufficiently large even number can be expressed as the sum of two `almost equal' primes, then we could prove the statement of \Cref{cor:FR} with a better $c=2$ instead. But as long as this strengthening of the binary Goldbach problem remains only a conjecture, we can only argue as follows. Baker, Harman, and  Pintz~\cite{BakHarPin2001} proved that for all sufficiently large $x$, there is a prime number between $x$ and $x+x^{0.525}$. In particular, their result implies that there is a prime $p_1$ close to $(r-s)/4$. Then we can write an odd number $(r-s)-p_1$ as the sum of three almost equal primes $(r-s)-p_1=p_2+p_3+p_4$ as earlier. Now we split both $n$ and $r$ into $4$ almost equal summands, and the rest of our argument goes as before completing the proof of \Cref{cor:FR}.

\vspace{3mm}

Note that every odd number can be expressed not only as the sum of three `almost equal' primes, but also as the sum of three primes in any given asymptotic proportion, see~\cite{Sag2019_ThreePrimes} for the details. So, one can consider asymmetric variations of the present technique, hoping to improve the base of the exponent in  \Cref{cor:FR}. Our numerical experiments suggest that this improvement is possible if and only if $\sigma < \rho/2$. For instance, if $r\sim0.5n$, $s \sim 0.15n$, and $r-s$ is odd, then any $s$-avoiding family $\F \subset \binom{[n]}{r}$ satisfies $|\F| \le \big(1.970+o(1)\big)^n$ by \Cref{cor:FR}, while the asymmetric partitions $n \sim 0.774n+0.113n+0.113n$, $r\sim 0.384n+0.062n+0.052n$, $s\sim 0.084n+0.038n + 0.028n$ give a better bound $|\F| \le \big(1.964+o(1)\big)^n$. However, both of these bounds are still far from the tight upper bound $|\F| \le \big(1.911+o(1)\big)^n$ via \Cref{th:FWPR}, valid in the case when $r-s$ is prime or a prime power.

\vspace{3mm}

Now let us implement basically the same ideas to prove \Cref{cor:clique}. First, we apply \Cref{lem:part} to find a partition $[kn]=N_1\cup \dots \cup N_k$ of $[kn]$ into $k$ disjoint parts of sizes $n$ such that the
subfamily
\begin{equation*}
	\F' \coloneqq \big\{F \in \F: |F\cap N_i| = r \mbox{ for all } i = [k]\big\}
\end{equation*}
satisfies
\begin{equation} \label{eq:clique2}
	|\F'|\cdot \binom{kn}{kr} \ge |\F|\cdot\binom{n}{r}^k.
\end{equation}

Given $i \in [k]$, let us consider the following graph $G_i=(V_i,E_i)$. Its set of vertices $V_i = \binom{N_i}{r}$ consists of all $r$-element subsets of $N_i$. An edge connects two vertices whenever they share exactly $s$ common elements. As earlier, it is clear that the independent sets are exactly the $s$-avoiding families by the definition. Hence, \Cref{th:FWPR,cor:FR} imply that
\begin{equation} \label{eq:clique4}
	\alpha(G_i) \le \binom{n}{r}\big(\delta(\rho,\sigma)+o(1)\big)^{n/c}.
\end{equation}

Observe that for all $F \in \F'$, we have $F\cap N_i \in \binom{N_i}{r}$ by the definition. Therefore, we can consider $\F'$ as a subset of the Cartesian product $\VV$ by corresponding each $F \in \F'$ to the $k$-tuple $(F\cap N_1, \dots, F\cap N_k)$.

Let us consider a collection of sets $F_0, F_1, \dots, F_k$ that correspond to some orthogonal star
\begin{equation*}
	\W_0 = (w_1,\dots,w_k), \ \W_1 =(u_1,w_2,\dots,w_k), \ \dots \ , \  \W_k =(w_1,\dots,w_{k-1},u_k).
\end{equation*}
in $\VV$, respectively. Note that for all $1\le j_1 < j_2 \le k$, we have
\begin{equation*}
	|F_{j_1} \cap F_{j_2}| = |w_{j_1}\cap u_{j_1}|+|w_{j_2}\cap u_{j_2}|+\sum_{\substack{i=1 \\ i\neq j_1,j_2}}^{k}|w_i| =2s+(k-2)r.
\end{equation*}
Since $\F$ contains no such collection by the assumption, then so does $\F'$. Hence $\F'$ also contains no orthogonal stars, and \Cref{lem:star} combined with~\eqref{eq:clique4} yields that
\begin{equation*}
	\frac{|\F'|}{\binom{n}{r}^k} \le k \big(\delta(\rho,\sigma)+o(1)\big)^{n/c} = \big(\delta(\rho,\sigma)+o(1)\big)^{n/c}.
\end{equation*}
Now~\eqref{eq:clique2} completes the proof of \Cref{cor:clique}.

\subsection{Weak sunflowers} \label{sec3:sun}

In the present section, we prove \Cref{cor:snflwr}. Given $k \ge 3$, let $\F \subset 2^{[n]}$ be a family of subsets of $[n]$ that does not contain weak $k$-sunflowers.

First, observe that the subfamily $\{F \in \F: |F|=r\}$ has size at least $\frac{|\F|}{n+1}$ for some $r \le n$ by the pigeonhole principle. Since this polynomial factor can be easily hidden into the $o(1)$-term of any exponential upper bound of the form $\big(c+o(1)\big)^n$, we can assume without loss of generality that $\F \subset \binom{[n]}{r}$. We can also assume that $2k$ divides $n$ by taking the smallest $n'>n$ that is a multiple of $2k$ and considering $\F$ as a family of subsets of $[n']\supset[n]$.

It is cleat that for all $F \in \F$ there are exactly $\binom{n}{n/2}$ subsets $G \subset [n]$ such that the cardinality of the symmetric differences $F\triangle G$ equals $n/2$. Hence, the pigeonhole principle implies the for some $G \subset [n]$, the family
\begin{equation*} \label{eq:sun1}
	\F' \coloneqq \big\{F\triangle G : F \in \F\big\}\cap \binom{[n]}{n/2}
\end{equation*}
satisfies
\begin{equation} \label{eq:sun2}
	|\F'|\cdot 2^n \ge |\F|\cdot\binom{n}{n/2}.
\end{equation}

Observe that $\F'$ does not contain weak $k$-sunflowers. Indeed, assume the contrary, namely that for some $F_1,\dots,F_k \in \F$, all the pairwise intersections between $F_1\triangle G,\dots,F_k \triangle G \in \F'$ are of the same cardinality. Then it is not hard to see that
\begin{equation*}
	2|F_i\cap F_j|-2|(F_i\triangle G)\cap(F_j \triangle G)| = |F_i|+|F_j|-|F_i\triangle G|- |F_j\triangle G| = 2r-n
\end{equation*}
for all $i \neq j$. In particular, this implies that the cardinality of $|F_i\cap F_j|$ is independent of $i$ and $j$. Thus the collection $F_1,\dots,F_k \in \F$ is a weak $k$-sunflower, a contradiction.

Let $p$ be the largest prime such that $p<\frac{2-\sqrt{2}}{4}\frac{n}{k}$. Baker, Harman, and  Pintz~\cite{BakHarPin2001} proved that for all sufficiently large $x$, there is a prime number between $x$ and $x+x^{0.525}$. In particular, their result implies that $p\sim\frac{2-\sqrt{2}}{4}\frac{n}{k}$ asymptotically for large $n$. Put $s = \frac{n}{2k}-p \sim \frac{\sqrt{2}}{4}\frac{n}{k}$.

Since $\F'$ contains no weak $k$-sunflowers, it also does not contain a collection of $k$ sets with pairwise intersections of cardinality $\frac{n}{2}-2p = 2s+(k-2)\frac{n}{2k}$. Hence, we can apply \Cref{cor:clique} with $\frac{n}{k}$ playing the role of $n$, $\frac{n}{2k}$ playing the role of $r$, and with $c=1$ to show that
\begin{equation} \label{eq:sun3}
	|\F'| \le \binom{n}{n/2}\big(\psi+o(1)\big)^{-n/k},
\end{equation}
where
\begin{equation*}
	\psi = \delta(1/2, \sqrt{2}/4)^{-1} = \exp\big\{H\big((2-\sqrt{2})/2\big)-H\big((2-\sqrt{2})/4\big)\big\} = \frac{1+\sqrt{2}}{2}.
\end{equation*}
Now we combine~\eqref{eq:sun2} with~\eqref{eq:sun3} to complete the proof of \Cref{cor:snflwr}.

\vspace{3mm}

Observe that this argument does not fully use the fact that $\F$ contains no weak $k$-sunflowers, but rather relies only on the corollary that $\F$ contains no collection of $k$ sets with a prescribed cardinality of pairwise intersections. Though this cardinality is optimal within the technique, there may be a way to utilize other forbidden intersections to improve the result.

\section{Open problems} \label{sec:conc}

{\bf 1.} Suppose that the dimension of the simplex $k=k(n)$ depends on $n$. The pigeonhole principle immediately implies that $\chi\big(\R^n, \triangle^{k(n)}\big)> \frac{n}{k(n)} \to \infty$ if $k(n)=o(n)$ as $n \to \infty$, and it is not hard to deduce some stronger lower  bounds from~\eqref{eq:chi_SK}. %that $\chi\big(\R^n, \triangle^{k(n)}\big)$ tends to infinity with $n$ for  %there exists a positive $\varepsilon$ such that for all sufficiently large $n$, we have $k(n) \le (\ln\psi_2-\varepsilon+o(1))\frac{n}{\ln n}$, where $\varepsilon>0$.
What happens for the larger values of $k$? For instance, does the value $\chi(\R^{2n}, \triangle^n)$ tend  to infinity with $n$?

{\bf 2.} For a fixed $k \in \N$, the best upper bound on the chromatic number $\chi\big(\R^n, \triangle^k\big)$ for large $n$ is due to Prosanov~\cite{Pros2018_ExpRams}: $\chi\big(\R^n, \triangle^k\big) \le \big(1+\sqrt{2(k+1)/k} +o(1)\big)^n$. Observe that the base of this exponent does not tend to $1$ as $k$ grows, unlike the base of the exponent in the lower bound from \Cref{cor:simpl}. This raises the following question. For an arbitrary small positive $\varepsilon$, is there a sufficiently large $k=k(\varepsilon)$ such that $\chi\big(\R^n, \triangle^k\big) \le \big(1+\varepsilon +o(1)\big)^n$ as $n \to \infty$?

{\bf 3.} Note that both problems above can be asked for non-Euclidean $\ell_p$-spaces as well. Let us give another question of this sort.

If  $\lambda_1=\dots=\lambda_k=1$ then let us denote the baton $\B(\lambda_1,\dots,\lambda_k)$ by $\B_k$ for a shorthand. Recall that $\chi(\R_1^n,\B_k) \ge \big(\psi_1^{1/k}+o(1)\big)^n$ by \Cref{cor:manh_sympl}, and the base of this exponent $\psi_1^{1/k} = 1+ \frac{0.3119...}{k}+ O\big(\frac{1}{k^2}\big)$ tends to $1$ with roughly `linear' speed. At the same time, the best upper bound that we can get by combining the ideas from~\cite{Kup} and~\cite{Pros2018_ExpRams} is only $\chi(\R_1^n,\B_k) \le \big(2+\frac{2}{k}+o(1)\big)^n$. It is then natural to ask if, for any $\varepsilon>0$, there exists a sufficiently large $k=k(\varepsilon)$ such that $\chi\big(\R_1^n, \B_k\big) \le \big(1+\varepsilon +o(1)\big)^n$.

{\bf 4.} Recall that for all $k\ge 3$, the best known lower bound on $G_k(n)$ given by \eqref{eq:snflwr1} is subexponential with $n$. Erd\H{o}s and Szemer\'edi wrote in~\cite{ErdSzem1978} that `there is a good chance' that the correct growth rate of $G_k(n)$ is indeed subexponential. As of now, the following much more modest question is still open. Is it true that, % there exists a sufficiently small positive $\varepsilon$ such that, 
for any $k \ge 3$, a family $\F \subset 2^{[n]}$ of size at least $1.99^n$ contains a weak $k$-sunflower, provided $n=n(k)$ is large enough?

\vspace{2mm}

\noindent {\bf \large Acknowledgments}. We thank Eric Naslund for bringing the reference \cite{Naslund23} to our attention.

\vspace{2mm}

\noindent {\bf \large Funding}. The results of this work obtained in Section~2 were supported by the  \href{https://rscf.ru/project/24-71-10021/}{RSF grant No. 24-71-10021}.
The results of this work obtained in Section~3 were in part supported by the ERC Advanced Grant `GeoScape' No. 882971.
%The second author is also a winner of Young Russian Mathematics Contest and would like to thank its sponsors and jury.

\end{document}